\documentclass[A4]{article}
\usepackage[hidelinks]{hyperref}
\usepackage{mathrsfs}
\usepackage{amsfonts}
\usepackage{latexsym}
\usepackage{epsfig}
\usepackage{indentfirst}
\graphicspath{{figures/}}
\usepackage{amsmath}
\usepackage{amsthm}
\usepackage{graphicx}
\usepackage{amssymb}
\usepackage{amscd}
\usepackage{latexsym}
\usepackage{subfigure}
\usepackage{epstopdf}
\usepackage{color}
\DeclareGraphicsExtensions{.eps}
\input xy
\xyoption{all}

\theoremstyle{definition}
\newtheorem{thm}{\textbf{Theorem} }[section]

\newtheorem{lem}[thm]{\textbf{Lemma}}

\newtheorem{cor}[thm]{\textbf{Corollary}}
\newtheorem{prop}[thm]{\textbf{Proposition}}
\newtheorem*{rmk}{\textbf{Remark}}
\newtheorem{defn}[thm]{\textbf{Definition}}

\numberwithin{equation}{section} \makeatletter
\renewenvironment{proof}[1][\proofname]{\par
    \pushQED{\qed}%
    \normalfont \topsep6\p@\@plus6\p@ \labelsep1em\relax
    \trivlist
    \item[\hskip\labelsep
        \bfseries #1]\ignorespaces
}{%
    \popQED\endtrivlist\@endpefalse
} \makeatother

\title{\bfseries \large{ON THE UPSILON INVARIANT OF CABLE KNOTS}}
\author{ \normalsize{Wenzhao Chen}\footnote{chenwenz@math.msu.edu}}
\date{}
\begin{document}
\maketitle
\parindent = 18pt

\begin{abstract}
In this paper, we study the behavior of $\Upsilon_K(t)$ under the cabling operation, where $\Upsilon_K(t)$ is the knot concordance invariant defined by Ozsv\'{a}th, Stipsicz, and Szab\'{o}, associated to a knot $K\subset S^3$. The main result is an inequality relating $\Upsilon_K(t)$ and $\Upsilon_{K_{p,q}}(t)$, which generalizes the inequalities of Hedden \cite{Hed09} and Van Cott \cite{VC10} on the Ozsv\'{a}th-Szab\'{o} $\tau$-invariant. As applications, we give a computation of $\Upsilon_{(T_{2,-3})_{2,2n+1}}(t)$ for $n\geq 8$, and we also show that the set of iterated $(p,1)$-cables of $Wh^{+}(T_{2,3})$ for any $p\geq 2$ span an infinite-rank summand of topologically slice knots. 
\end{abstract}

\section{Introduction}
The complete knot Floer chain complex $CFK^{\infty}(K)$ is a bifiltered, Maslov graded chain complex associated to a knot $K\subset S^3$, introduced by Ozsv\'{a}th and Szab\'{o} \cite{OS04c}, and independently by Rasmussen \cite{Ras03}. \textit{A priori}, the bifiltered chain homotopy type of $CFK^{\infty}(K)$ is an isotopy invariant of the knot $K$. By exploiting the TQFT-like aspects of the theory, however, it turns out that $CFK^{\infty}(K)$ also contains a lot of interesting information about the concordance class of $K$, see \cite{Hom15} for a survey. One typical example is the Ozsv\'{a}th-Szab\'{o} $\tau$-invariant, which came to stage relatively early and attracted a lot of attention \cite{OS03}. Roughly speaking, $\tau(K)$ is a concordance homomorphism that takes its values in $\mathbb{Z}$ and is defined by examining $\widehat{CFK}(K)$, which is a small portion of $CFK^{\infty}(K)$. Moreover, $\tau$ provides a lower bound to the smooth four-genus of a knot i.e. $|\tau(K)|\leq g_4(K)$. Among many applications, Oszv\'{a}th and Szab\'{o} showed that $\tau$ can be used to resolve a conjecture of Milnor first proven by Kronheimer and Mrowka using gauge theory \cite{KM93}, that $g_4(T_{p,q})=\frac{(p-1)(q-1)}{2}$, where $T_{p,q}$ is the $(p,q)$-torus knot. 

Recently, by using more information from $CFK^{\infty}(K)$, Ozsv\'{a}th, Stipsicz, and Szab\'{o} introduced a more powerful concordance invariant that generalizes $\tau$ \cite{OSS}. This invariant takes the form of a homomorphism from the smooth knot concordance group $\mathcal{C}$ to the group of piecewise linear functions on $[0,2]$. Thus, for every knot $K\subset S^3$, they associate a piecewise linear function $\Upsilon_K(t)$, where $t\in[0,2]$ that depends only on the concordance type of $K$.

Besides being a concordance homomorphism, $\Upsilon_K(t)$ also enjoys many other nice properties, some of which we list below.
\begin{itemize}
\item[(1)] ({Symmetry}) $\Upsilon_K(t)=\Upsilon_K(2-t),$
\item[(2)] (4-genus bound) $|\Upsilon_K(t)|\leq tg_4(K)$ for $0\leq t\leq 1$,
\item[(3)] (Recovers $\tau$) The slope of $\Upsilon_K(t)$ at $t=0$ is $-\tau(K)$.
\end{itemize}

In this paper we study how $\Upsilon$ behaves under the cabling operation. Such results for $\tau$ can be found in \cite{Hed05a,Hed05b,Hed09,Hom14,Pet13,VC10}, among which we restate two results. The first one is due to Hedden, obtained by carefully comparing the knot Floer chain complex of a knot and that of its cable.
\begin{thm}(\cite{Hed09})\label{Hed}Let $K\subset S^3$ be a knot, and $p>0$, $n \in \mathbb{Z}$. Then 
\begin{displaymath}
p\tau(K)+\frac{pn(p-1)}{2}\leq \tau(K_{p,pn+1})\leq p\tau(K)+\frac{(pn+2)(p-1)}{2},
\end{displaymath}
\end{thm}

Later Van Cott used the genus bound and homomorphism property satisfied by $\tau$, together with nice constructions of cobordism between cable knots to extend the above inequality to $(p,q)$-cables. 
\begin{thm}(\cite{VC10})\label{VC}
Let $K\subset S^3$ be a knot, and  $(p,q)$ be a pair of relatively prime numbers such that $p>0$. Then 
\begin{displaymath}
p\tau(K)+\frac{(p-1)(q-1)}{2}\leq \tau(K_{p,q})\leq p\tau(K)+\frac{(p-1)(q+1)}{2},
\end{displaymath}
\end{thm}

In view of the success in understanding the effect of knot cabling on the $\tau$-invariant, it is natural to wonder what happens to $\Upsilon$. In this paper we show that a portion of $\Upsilon$ behaves very similarly to $\tau$. Indeed, adapting the strategies of Hedden and Van Cott on studying $\tau$ to the context of $\Upsilon$, we can prove the following result:
\begin{thm}  \label{main}
Let $K\subset S^3$ be a knot, and  $(p,q)$ be a pair of relatively prime numbers such that $p>0$. Then
\begin{displaymath}
\Upsilon_K(pt)-\frac{(p-1)(q+1)t}{2}\leq \Upsilon_{K_{p,q}}(t)\leq \Upsilon_K(pt)-\frac{(p-1)(q-1)t}{2},
\end{displaymath}
when $0 \leq t\leq \frac{2}{p}$.
\end{thm}

Note that by differentiating the above inequality at $t=0$, we recover Theorem \ref{VC}.  One also easily sees this inequality is sharp by examining the case when $K$ is the unknot: when $q>0$, then the upper bound is achieved, and when $q<0$ the lower bound is achieved. However, when $p>2$ the behavior of $\Upsilon_{K_{p,q}}(t)$ for $t\in[\frac{2}{p},2-\frac{2}{p}]$ is still unknown to the author.\\

Theorem \ref{main} can often be used to determine the $\Upsilon$ function of cables with limited knowledge of their complete knot Floer chain complexes.  For an example, we show how our theorem can be used to deduce $\Upsilon_{(T_{2,-3})_{2,2n+1}}(t)$ for $n\geq 8$; these examples are, on the face of it, rather difficult to compute, since none of them is an $L$-space knot.  Despite this, armed only with Theorem \ref{main} and the knot Floer homology groups of the knots in this family, we are able to obtain complete knowledge of $\Upsilon$.  

Perhaps more striking, however, Theorem \ref{main} can be used to easily show that certain subsets of the smooth concordance group freely generate infinite-rank summands.  To this end, let  $D=Wh^{+}(T_{2,3})$ be the untwisted positive whitehead double of the trefoil knot and let $J_n=((D_{p,1})...)_{p,1}$ denote the $n$-fold iterated $(p,1)$-cable of $D$ for some fixed $p>1$ and some positive integer $n$. Theorem \ref{main}, together with general properties of $\Upsilon$, yield the following
\begin{cor}\label{cor}
The family of knots $J_n$ for $n=1,2,3,...$ are linearly independent in $\mathcal{C}$ and span an infinite-rank summand consisting of topologically slice knots.
\end{cor}

To the best of the author's knowledge, Corollary \ref{cor} provides the  first satellite operator on the smooth concordance group of topologically slice knots whose iterates (for a fixed knot) are known to be independent; moreover, in this case they are a summand.  Note $J_n$ has trivial Alexander polynomial.  The first known example of infinite-rank summand of knots with trivial Alexander polynomial is generated by $D_{n,1}$ for $n\in \mathbb{Z}^+$, due to Kim and Park \cite{KP16}.   Their example, however, like the families of topologically slice knots studied by \cite{OSS}, involved rather non-trivial calculuatons of $\Upsilon$ (for instance, those of \cite{OSS} involved rather technical caclulations with the bordered Floer invariants).    The utility of Theorem \ref{main} is highlighted by the ease with which the above family is handled. We also refer the interested reader to \cite{FPR16} for a host of other applications of Theorem \ref{main} to the study of the knot concordance group.\\

To conclude the introduction, it is worth mentioning that Hom achieved a complete understanding of the behavior of $\tau$ under cabling by introducing the $\epsilon$-invariant \cite{Hom14}. In particular, $\tau(K_{p,q})$ is always one of the two bounds appearing in Theorem \ref{VC}, depending on the value of $\epsilon(K)$. However, in the context of $\Upsilon$ the story is not true, even for $L$-space knots whose knot Floer chain complexes are relatively simple. For instance, $\Upsilon_{(T_{2,3})_{n,2n-1}}(t)$ is computed in \cite{OSS} but it does not equal either bound appearing in Theorem \ref{main}. This suggests the behavior of $\Upsilon$ under cabling is more complicated. For example, it would be interesting if one can find suitable auxiliary invariants serving a similar role as $\epsilon$-invariant.\\

\textbf{Outline.} The organization of the rest of the paper is as following: in Section 2, we review the definition of $\Upsilon$. In Section 3, we prove our main theorem. In Section 4 we give two applications of Theorem \ref{main}. 

\subsection*{Acknowledgments.} I wish to thank my advisor Matt Hedden, for teaching me with great patience, and lending me his vision through wonderful suggestions that shape this work. I also thank Andrew Donald and Kouki Sato for many nice comments on earlier versions of this paper.

\section{Preliminaries}
We work over $\mathbb{F}=\mathbb{Z}/2\mathbb{Z}$ throughout the entire paper. We also assume that the reader is familiar with the basic setup of knot Floer homology. For more details, see \cite{OS04c,Ras03}.

In this section, we will briefly review the construction of $\Upsilon_K(t)$ of a given knot $K$, setting up some notations at the same time. The original definition of $\Upsilon_K(t)$ is based on a $t$-modified knot Floer chain complex, see \cite{OSS}. Shortly thereafter, Livingston reformulated $\Upsilon_K(t)$ in terms of the complete knot Floer chain complex $CFK^{\infty}(K)$. We find it convenient to work with Livingston's definition, which we recall below.

Denote $CFK^{\infty}(K)$ by $C(K)$ for convenience. Note that $C(K)$ comes with a $\mathbb{Z}\oplus\mathbb{Z}$-filtration, namely the Alexander filtration and the algebraic filtration. Actually, to be more precise, $C(K)$ is only well defined up to bifiltered chain homotopy equivalence, unless we fix some compatible Heegaard diagram for $K$ and some auxiliary data. 

Now for any $t\in[0,2]$, one can define a filtration on $C(K)$ as follows. First, define a real-valued (grading) function on $C(K)$ by $$\mathcal{F}_t=\frac{t}{2}Alex+(1-\frac{t}{2})Alg,$$ which is a convex linear combination of Alexander and algebraic gradings. Associated to this function, one can construct a filtration given by $(C(K), \mathcal{F}_t)_s=(\mathcal{F}_t)^{-1}(-\infty,s]$. It is easy to see that the filtration induced by $\mathcal{F}_t$ is compatible with the differential of $C(K)$, i.e. $\mathcal{F}_t(\partial x)\leq \mathcal{F}_t(x)$, $\forall x\in C(K)$. Let $$\nu(C(K), \mathcal{F}_t)=\min\{s\in\mathbb{R}|H_{0}((C(K), \mathcal{F}_t)_s)\rightarrow H_{0}(C(K))\ \text{is\ nontrivial}\}.$$
Here $H_{0}$ stands for the homology group with Maslov grading $0$. With these preparations, $\Upsilon$ is defined as following.
\begin{defn}
$\Upsilon_K(t)=-2\nu(C(K), \mathcal{F}_t)$.
\end{defn}

It is proven in \cite{Liv14} that the above definition of $\Upsilon_K(t)$  is equivalent to the one given by Ozsv\'{a}th, Stipsicz, and Szab\'{o} in \cite{OSS}.

\section{Proof of the main theorem}
The proof of Theorem \ref{main} will be divided into two parts: in Subsection 3.1 we will prove the inequality for the $(p,pn+1)$ cable of a knot by adapting Hedden's strategy in \cite{Hed05a,Hed09}, and then in Subsection 3.2 we will upgrade the inequality to cover the $(p,q)$-cable of a knot by applying Van Cott's argument in \cite{VC10}. 
\subsection{Upsilon of $(p,pn+1)$-cable}
Following \cite{Hed05a,Hed05b,Hed09}, we will begin with introducing a nice Heegaard diagram which encodes both the original knot $K$ and its cable $K_{p,pn+1}$.

For any knot $K\subset S^3$, there exists a compatible Heegaard diagram $H=(\Sigma, \{\alpha_1,...,\alpha_g\},\{\beta_1,...,\beta_{g-1},\mu\},w,z)$. Moreover, by stabilizing we can assume $\mu$ to be the meridian of the knot $K$ and  that it only intersects $\alpha_g$, and there is a $0$-framed longitude of the knot $K$ on $\Sigma$ which does not intersect $\alpha_g$. From now on, we will always assume that the Heegaard diagram $H$ for $K$ satisfies all these properties.

Let $H=(\Sigma, \{\alpha_1,...,\alpha_g\},\{\beta_1,...,\beta_{g-1},\mu\},w,z)$ be a Heegaard diagram for the knot $K$ as above. By modifying $\mu$ and adding an extra base point $z'$, we can construct a new Heegaard diagram with three base points $H(p,n)=(\Sigma, \{\alpha_1,...,\alpha_g\},\{\beta_1,...,\beta_{g-1},\tilde{\beta}\},w,z,z')$. More precisely, $\tilde{\beta}$ is obtained by winding $\mu$ along an $n$-framed longitude $(p-1)$ times, and the new base point $z'$ is placed at the tip of the winding region such that the arc $\delta'$ connecting $w$ and $z'$ has intersection number $p$ with $\tilde{\beta}$. Note $\tilde{\beta}$ can be deformed to $\mu$ through an isotopy that does not cross the base points $\{w,z\}$. See Figure~\ref{longitude} and Figure~\ref{winding} for an example. The power of $H(p,n)$ lies in the fact that it specifies both $K$ and $K_{p,pn+1}$ at the same time, as pointed out by the following lemma.

\begin{lem}(Lemma 2.2 of \cite{Hed05a})
Let $H(p,n)$ be a Heegaard diagram described as above. Then
\begin{enumerate}
\item Ignoring $z'$, we get a doubly-pointed diagram $H(p,n,w,z)$ which specifies $K$.
\item Ignoring $z$, we get a doubly-pointed diagram $H(p,n,w,z')$ which specifies the cable knot $K_{p,pn+1}$.
\end{enumerate}

\end{lem}

This implies that the two knot Floer chain complexes $CFK^\infty(H(p,n,w,z))$ and $CFK^\infty(H(p,n,w,z'))$ are closely related. More precisely, by forgetting the Alexander filtrations, both $CFK^\infty(H(p,n,w,z))$ and $CFK^\infty(H(p,n,w,z'))$ are isomorphic to $CF^\infty(H(p,n,w))$. Therefore, in order to get a more transparent correspondence between these two complexes, we will compare the Alexander gradings of the intersection points with respect to the two different base points $z$ and $z'$.

For the sake of a clearer discussion, we fix some notation and terminology to deal with the intersection points. For convenience, we assume $n\geq 0$ through out the discussion and remark that the case when $n<0$ can be handled in a similar way. Note $\tilde{\beta}$ intersects $\alpha_g$ at $2(p-1)n+1$ points, and we label them as $x_0$, ..., $x_{2(p-1)n}$, starting at the out-most layer from left to right, and then the second layer from left to right, and so on. On the other hand, $\tilde{\beta}$ could also intersect other $\alpha$-curves besides $\alpha_g$, and we label these points by $y^{(k)}_{0}$,..., $y^{(k)}_{2(p-1)-1}$. Here $k$ enumerates the intersections of the $n$-framed longitude with $\alpha_i$, $i\neq g$, and the order of this enumeration is irrelevant. The lower index is again ordered following a layer by layer convention, from outside to inside, but we require that \textit{$y^{(k)}_0$ can be connected to $x_{2n}$ by an arc on $\tilde{\beta}$ which neither intersects $\delta$ nor $\delta'$}, the short arcs connecting the base points. See Figure~\ref{winding} for an example. The generators will be partitioned into $p$ classes: all the generators of the form $\{x_{2i},\textbf{a}\}$ or $\{y^{(k)}_{2i},\textbf{b}\}$ will be called \emph{even intersection points} or \emph{$0$-intersection points}, and \emph{odd intersection points} otherwise; odd generators of the form $\{x_{2i+1},\textbf{a}\}$ or $\{y^{(k)}_{2\lceil \frac{i+1}{n}\rceil-1},\textbf{b}\}$ will be called \emph{$(p-\lceil \frac{i+1}{n}\rceil)$-intersection points}. Here $\textbf{a},\textbf{b}$ are $(g-1)$-tuple in $Sym^{g-1}(\Sigma)$. Note that essentially we are classifying odd intersection points into $(p-1)$ classes by the following principle: if its $\tilde{\beta}$-component sits on the $i$-th layer (we count the layers from outside to inside), then it is called a $(p-i)$-intersection point.

\begin{figure}[!ht]
\centering{
\resizebox{110mm}{!}{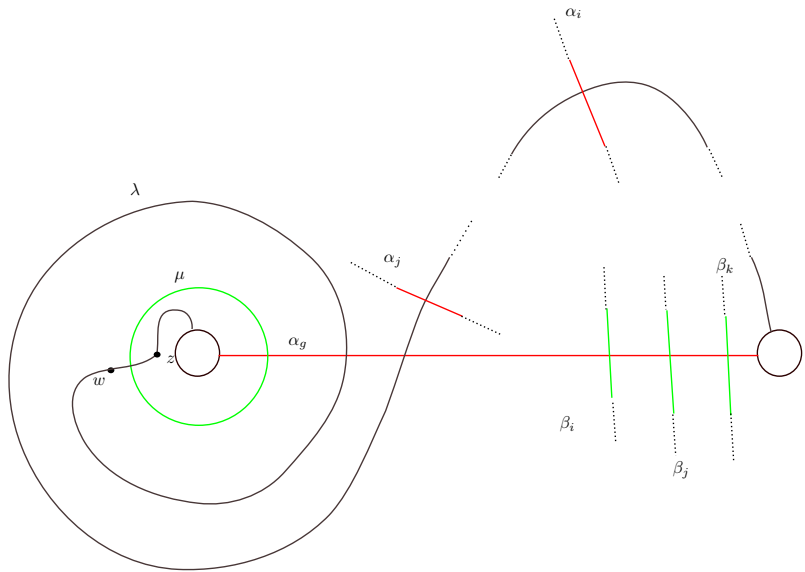}
\caption{A compatible Heegaard diagram $H$ for $K$. $\lambda$ is a $2$-framed longitude, and according to our assumption that the $0$-framed longitude can be chosen not to hit $\alpha_g$, $\lambda$ can be chosen to intersect $\alpha_g$ twice.}\label{longitude}
}
\end{figure}
\begin{figure}[!ht]
\centering{
\resizebox{125mm}{!}{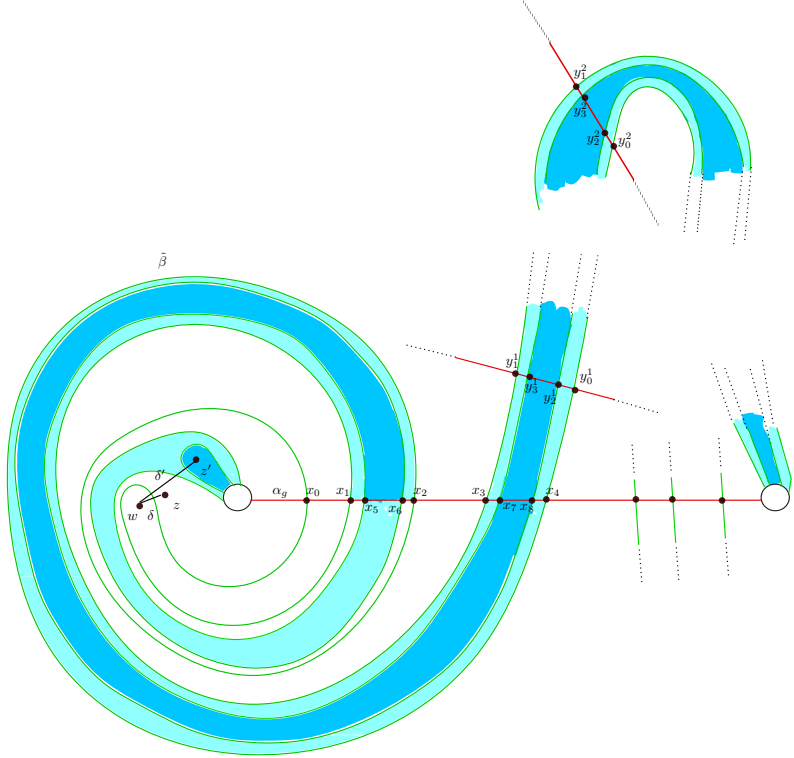}
\caption{A example of $H(p,n)$ with $n=2$ and $p=3$, corresponding to the Heegaard diagram shown in Figure~\ref{longitude}. There is an obvious arc of $\tilde{\beta}$ connecting $x_4$ and $y^{1}_0$, which neither intersects $\delta$ nor $\delta'$. By our convention, there is an arc of $\tilde{\beta}$ connecting $x_4$ and $y^{2}_0$ satisfying the same property as well, though it is not shown in the figure. The shaded region represents a domain connecting $\{x_1,\textbf{a}\}$ and $\{x_2,\textbf{a}\}$; the darkened color indicates the multiplicity is 2, while the lighter colored region has multiplicity 1. }\label{winding}
}
\end{figure}
We denote the Alexander grading by $A$ (by $A'$) when we use the base point $z$ (base point $z'$).
The comparison of Alexander filtrations is summarized in the following proposition.
\begin{prop}\label{comparison}
With the choice of Heegaard diagrams as described above and let $\textbf{x}$ be an $l$-intersection point, where $l\in\{0,1,...,p-1\}$, then
$$A'(\textbf{x})=pA(\textbf{x})+\frac{pn(p-1)}{2}+l.$$
\end{prop}

Proposition \ref{comparison} can be viewed as a generalization of the comparison used in\cite{Hed05a,Hed09}, in which only $\{x_i,\textbf{a}\}$ for $i\leq n$ were shown to satisfy the above equation. In studying $\tau$, having just a comparison for $\{x_i,\textbf{a}\}$ for $i\leq n$ would suffice: first, Hedden observed that in the case when $|n|$ is sufficiently large they account for the top Alexander graded generators of $\widehat{CFK}(K_{p,pn+1})$ that determine $\tau(K_{p,pn+1})$; second, the behavior of $\tau$ for small $n$ can be deduced from the large-$n$ case by using crossing change inequality of $\tau$. In contrast, the lower Alexander graded elements of $CFK^\infty(K_{p,pn+1})$ may play a role in $\Upsilon$, even though they do not affect $\tau$. Therefore in the current paper we have to carry out a comparison for all types of generators. To accomplish this goal, we quote and extend some of the lemmas used in \cite{Hed05a,Hed09} below, after which Proposition \ref{comparison} will follow easily. 

\begin{lem}
When $1\leq j\leq (p-1)n$, we have
\begin{equation}
 A(\{x_{2j-1},\textbf{a}\})-A(\{x_{2j},\textbf{a}\})=0
\end{equation}
\begin{equation}
 A'(\{x_{2j-1},\textbf{a}\})-A'(\{x_{2j},\textbf{a}\})=p-\lceil\frac{j}{n}\rceil
\end{equation}

For an arbitrary $k$, when  $0\leq i\leq (p-2)$, we have
\begin{equation}
 A(\{y^{(k)}_{2i+1},\textbf{a}\})-A(\{y^{(k)}_{2i},\textbf{a}\})=0
\end{equation}
\begin{equation}
 A'(\{y^{(k)}_{2i+1},\textbf{a}\})-A'(\{y^{(k)}_{2i},\textbf{a}\})=p-(i+1) 
\end{equation}
\end{lem}
\begin{proof}
Note that there is a Whitney disk $\phi$ connecting $\{x_{2j-1},\textbf{a}\}$ to $\{x_{2j},\textbf{a}\}$ (See Figure~\ref{winding}).  It is the product of a constant map in $Sym^{g-1}(\Sigma)$ and the map represented by the disk which connects $x_{2j-1}$ and $x_{2j}$, with boundary consisting of a short arc of $\alpha_g$ and an arc of $\tilde{\beta}$ that spirals into the winding region $p-\lceil\frac{j}{n}\rceil$ times and then makes a turn out. We can see that $n_w(\phi)=n_z(\phi)=0$ and $n_{z'}(\phi)=p-\lceil\frac{j}{n}\rceil$. Therefore, $A(\{x_{2j-1},\textbf{a}\})-A(\{x_{2j},\textbf{a}\})=n_z(\phi)-n_w(\phi)=0$ and $A'(\{x_{2j-1},\textbf{a}\})-A'(\{x_{2j},\textbf{a}\})=n_{z'}(\phi)-n_w(\phi)=p-\lceil\frac{j}{n}\rceil$. We have obtained equation (3.1) and (3.2). The proof for (3.3) and (3.4) will follow a similar line, and hence is omitted. \\
\end{proof}
\begin{lem}
When $0\leq j\leq (p-1)n$, we have
\begin{equation}
 p(A(\{x_{0},\textbf{a}\})-A(\{x_{2j},\textbf{a}\}))=A'(\{x_{0},\textbf{a}\})-A'(\{x_{2j},\textbf{a}\})
\end{equation}

For an arbitrary $k$, when  $0\leq i\leq (p-2)$, we have
\begin{equation}
 p(A(\{y^{(k)}_{0},\textbf{b}\})-A(\{y^{(k)}_{2i},\textbf{b}\}))=A'(\{y^{(k)}_{0},\textbf{b}\})-A'(\{y^{(k)}_{2i},\textbf{b}\})
\end{equation}

\begin{equation}
 p(A(\{x_{2n},\textbf{a}\})-A(\{y^{(k)}_0,\textbf{b}\}))=A'(\{x_{2n},\textbf{a}\})-A'(\{y^{(k)}_0,\textbf{b}\}).
\end{equation}
\end{lem}
\begin{figure}[!ht]
\centering{
\resizebox{110mm}{!}{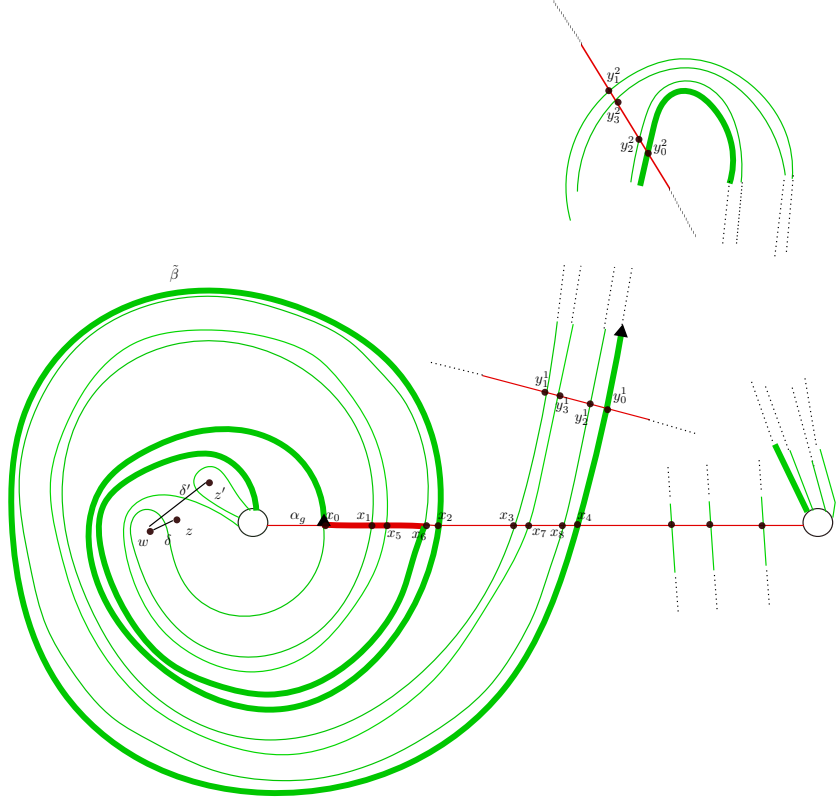}
\caption{The thickened curve $\gamma$ represents the $\epsilon$-class between $\{x_0,\textbf{a}\}$ and $\{x_6,\textbf{a}\}$. Note that the arc $\delta$ and $\delta'$ which connect based points do not intersect $\gamma$.}\label{epsilonclass1}
}
\end{figure}
\begin{figure}[!ht]
\centering{
\resizebox{110mm}{!}{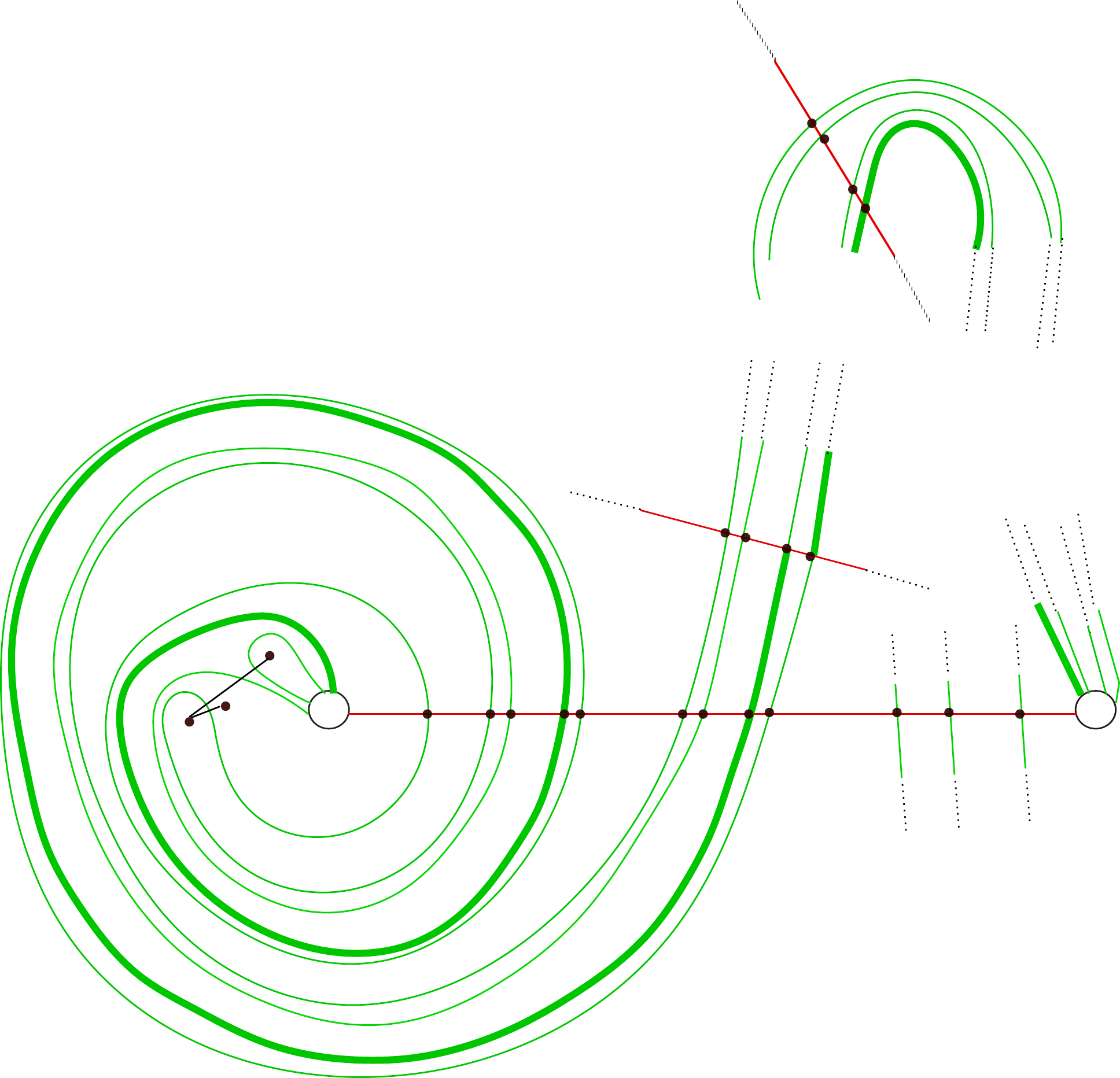}
\caption{The thickened curve is an arc on $\tilde{\beta}$ connecting $y^{(1)}_{0}$ and $y^{(1)}_{2}$ that does not intersect $\delta$ nor $\delta'$.}\label{epsilonclass2}
}
\end{figure}
\begin{proof}
First we prove Equation (3.5). Note that $\epsilon(\{x_{2j},\textbf{a}\},\{x_{0},\textbf{a}\})$ can be represented by a curve $\gamma$ on $\Sigma$, which is obtained by first connecting $x_{2j}$ to $x_0$ along $\alpha_g$, and then by an arc on $\tilde{\beta}$ which starts from $x_0$ and winds $j$ times counterclockwise to arrive at $x_{2j}$ (Figure~\ref{epsilonclass1}).  Note that $[\epsilon(\{x_{2j},\textbf{a}\},\{x_{0},\textbf{a}\})]=0\in H_1(S^3,\mathbb{Z})$, hence $[\gamma]=\Sigma l_i\alpha_i+\Sigma k_i \beta_i$, with $\beta_g$ is viewed as $\tilde{\beta}$. Let $c=\gamma-\Sigma l_i\alpha_i-\Sigma k_i \beta_i$, then $c$ bounds a domain on $\Sigma$. Note that since $\delta\cdot \gamma = \delta' \cdot \gamma=0$, we have $\delta'\cdot c=  \delta'\cdot (-k_g \tilde{\beta})= -k_g p = p (\delta\cdot (-k_g \tilde{\beta})) =p(\delta\cdot c)$, where ``$\cdot$'' stands for the intersection number. Equation (3.5) follows. \\
The proofs for the other two equations follow a similar line. Note the key point in the above argument is that the $\epsilon$-class of the two generators can be represented by a curve $\gamma$ whose arc on $\tilde{\beta}$ does not intersect the arc $\delta$ nor $\delta'$, implying $\delta\cdot \gamma = \delta' \cdot \gamma=0$. For Equation (3.6), note $y^{(k)}_{0}$ and $y^{(k)}_{2i}$ can be joined by an arc on $\tilde{\beta}$ satisfying the aforementioned property (see Figure~\ref{epsilonclass2} for an example). Recall by our convention, $y^{(k)}_0$ can be connected to $x_{2n}$ by an arc on $\tilde{\beta}$ which neither intersects $\delta$ nor $\delta'$, hence Equation (3.7) follows.
\end{proof}

Let $C(i)=\{(g-1)-tuples\textit{ } \textbf{a}|\textit{ }A(\{x_0,\textbf{a}\})=i\}$.
\begin{lem}(Lemma 3.4 of \cite{Hed05a})
Let $\textbf{a}_1\in C(j_1)$ and $\textbf{a}_2\in C(j_2)$, then
\begin{displaymath}
\begin{aligned}
 A(\{x_{i},\textbf{$\textbf{a}_1$}\})-A(\{x_{i},\textbf{$\textbf{a}_2$}\})=j_1-j_2\\
 A'(\{x_{i},\textbf{$\textbf{a}_1$}\})-A'(\{x_{i},\textbf{$\textbf{a}_2$}\})=p(j_1-j_2).
\end{aligned}
\end{displaymath}
\end{lem}

Now we are ready to prove Proposition 3.2.
\begin{proof}[Proof of Proposition 3.2.]
We want to prove that if $\textbf{x}$ an $l$-intersection point, where $l\in\{0,1,...,p-1\}$, then
$A'(\textbf{x})=pA(\textbf{x})+\frac{pn(p-1)}{2}+l$. As pointed out in Lemma 2.5 of \cite{Hed09}, $A'(\{x_0,\textbf{a}\})=pA(\{x_0,\textbf{a}\})+\frac{pn(p-1)}{2}$. Note that for any other intersection point $\textbf{u}$, as long as $A'(\{x_0,\textbf{a}\})-A'(\textbf{u})=p(A(\{x_0,\textbf{a}\})-A(\textbf{u}))$, then we have $A'(\textbf{u})=pA(\textbf{u})+\frac{pn(p-1)}{2}$ as well. Now the case when $l=0$ (even intersection points) follows easily from this obersvation, Lemma 3.4, and Lemma 3.5. The other cases is then an easy consequence of the $l=0$ case and Lemma 3.3.
\end{proof}

Let $C=CF^\infty(H(p,n,w))$ be the chain complex obtained by forgetting the Alexander filtraion. And let $(\mathcal{F}_{t})=\frac{t}{2} A+(1-\frac{t}{2}) Alg$, and $(\mathcal{F}_{t})'=\frac{t}{2} A'+(1-\frac{t}{2}) Alg$ be two grading functions on $C$ defined by using the two Alexander gradings $A$ and $A'$ corresponding to $z$ and $z'$ respectively. Then the filtrations corresponding to $(\mathcal{F}_{t})$ and $(\mathcal{F}_{t})'$ satisfy the following relation.

\begin{lem} For $p, n\in \mathbb{Z}$, $p>0$, and $0\leq t\leq \frac{2}{p}$, we have
$$(C, \mathcal{F}^{'}_{t})_{s+\frac{pn(p-1)t}{4}}\subset (C, \mathcal{F}_{pt})_s \subset (C, \mathcal{F}^{'}_{t})_{s+\frac{(pn+2)(p-1)t}{4}}.$$
\end{lem}
\begin{proof}
Let $\textbf{x}$ be a generator of $C$. Assume $U^{-k}\textbf{x}\in (C,\mathcal{F}_{pt})_s$, then $$ \frac{pt}{2}(A(\textbf{x})+k)+(1-\frac{pt}{2})k =\frac{pt}{2}A(\textbf{x})+k \leq s.$$
Combine the above inequality with Prop. 3.2, we have 
\begin{displaymath}
\begin{aligned}
&\frac{t}{2}(A'(\textbf{x})+k)+(1-\frac{t}{2})k\\
\leq &\frac{t}{2}(pA(\textbf{x})+\frac{pn(p-1)}{2}+p-1+k)+(1-\frac{t}{2})k\\
=&\frac{pt}{2}A(\textbf{x})+k+\frac{t}{2}\frac{pn(p-1)}{2}+\frac{t}{2}(p-1)\\
\leq &s+\frac{(pn+2)(p-1)t}{4}.
\end{aligned}
\end{displaymath}
Hence $U^{-k}\textbf{x}\in (C,\mathcal{F}^{'}_{t})_{s+\frac{(pn+2)(p-1)t}{4}}$, and therefore $$(C, \mathcal{F}_{pt})_s \subset (C, \mathcal{F}^{'}_{t})_{s+\frac{(pn+2)(p-1)t}{4}}.$$

Similarly, if we assume $U^{-k}\textbf{x}\notin (C,\mathcal{F}_{pt})_s$, then $$ \frac{pt}{2}(A(\textbf{x})+k)+(1-\frac{pt}{2})k > s.$$
Again, in view of the above inequality and Prop. 3.2, we have
\begin{displaymath}
\begin{aligned}
&\frac{t}{2}(A'(\textbf{x})+k)+(1-\frac{t}{2})k\\
\geq &\frac{t}{2}(pA(\textbf{x})+\frac{pn(p-1)}{2}+k)+(1-\frac{t}{2})k\\
=&\frac{pt}{2}A(\textbf{x})+k+\frac{t}{2}\frac{pn(p-1)}{2}\\
> &s+\frac{pn(p-1)t}{4}
\end{aligned}
\end{displaymath}
Hence $U^{-k}\textbf{x}\notin (C,\mathcal{F}^{'}_{t})_{s+\frac{pn(p-1)t}{4}}$, and therefore $$(C,\mathcal{F}^{'}_{t})_{s+\frac{pn(p-1)t}{4}}\subset (C,\mathcal{F}_{pt})_s.$$
\end{proof}

\begin{proof}[Proof of Theorem \ref{main} for $(p,pn+1)$-cable.] Recall that $$\nu(C,\mathcal{F}_t)=\min\{s|\ H_{0}((C, \mathcal{F}_t)_s)\rightarrow H_{0}(C)\ \text{is \ nontrivial} \},$$
and $\nu(C,{\mathcal{F}_t}^{'})$ is understood similarly. Now set $s=\nu(C,\mathcal{F}_{pt})$ in lemma 3.6, we have $$\nu(C,\mathcal{F}_{pt})+\frac{pn(p-1)t}{4} \leq\nu(C,{\mathcal{F}_t}^{'})\leq \nu(C,\mathcal{F}_{pt})+\frac{(pn+2)(p-1)t}{4}.$$
Recall $\Upsilon_K(pt)=-2\nu(C,\mathcal{F}_{pt})$ and $\Upsilon_{K_{p,pn+1}}(t)=-2\nu(C,\mathcal{F}'_{t})$, so by mutiplying $-2$ the above inequality translates to$$\Upsilon_K(pt)-\frac{(pn+2)(p-1)t}{2}\leq \Upsilon_{K_{p,pn+1}}(t)\leq \Upsilon_K(pt)-\frac{pn(p-1)t}{2}.$$
\end{proof}

\subsection{Upsilon of $(p,q)$-cable}
Denote the smooth knot concordance group by $\mathcal{C}$. Let $\theta : \mathcal{C} \rightarrow \mathbb{R}$ be a concordance homomorphism such that $|\theta(K)|\leq g_4(K)$ and $\theta(T_{p,q})=\frac{(p-1)(q-1)}{2}$ when $p,q>0$. In \cite{VC10}, Van Cott proved that if we fix a knot $K$ and $p>0$, and let $$h(l)=\theta(K_{p,l})-\frac{(p-1)l}{2},$$ then we have
\begin{equation}
-(p-1)\leq h(n)-h(r)\leq 0,
\end{equation} 
when $n>r$ such that both $n$ and $r$ relatively prime to $p$.
\begin{rmk}
The concordance homorphism studied by Van Cott has range $\mathbb{Z}$ rather than $\mathbb{R}$, but by checking the argument in \cite{VC10}, it is straightforward to see that this choice will not affect the inequality stated above. 
\end{rmk}
Now note that fixing $t\in (0, 2/p]$, $\frac{-\Upsilon_{K}(t)}{t}$ is a concordance homorphism which lower bounds the four-genus, and $\frac{-\Upsilon_{T_{p,q}}(t)}{t}=\frac{(p-1)(q-1)}{2}$ when $q>0$ (\cite{LVC15}). So we can take $\theta$ to be $\frac{-\Upsilon_{K}(t)}{t}$ and apply inequality (3.8), from which we get 
\begin{equation}
0\leq \bar{h}(n,t)-\bar{h}(r,t)\leq (p-1)t,
\end{equation}
where $\bar{h}(n,t)=\Upsilon_{K_{p,n}}(t)+\frac{(p-1)nt}{2}$. It is easy to see inequality (3.9) is true at $t=0$ as well, and hence it holds for $0\leq t\leq \frac{2}{p}$.

Following essentially the argument of Corollary 3 in \cite{VC10}, we conclude the proof of our main theorem as below:
\begin{proof}[Proof of Theorem \ref{main} for $(p,q)$-cable.]
Recall $0\leq t \leq \frac{2}{p}$. First we will show that $$\Upsilon_{K_{p,q}}(t)\geq \Upsilon_K(pt)-\frac{(p-1)(q+1)t}{2}.$$ Take $r$ to be any integer such that $q\geq pr+1$, then by inequality (3.9) we have $$\bar{h}(q,t)-\bar{h}(pr+1,t)\geq 0.$$ 
In view of the definition of $\bar{h}$, the above inequality translates to
\begin{equation}
\Upsilon_{K_{p,q}}(t) \geq \Upsilon_{K_{p,pr+1}}(t)-\frac{(p-1)(q-pr-1)t}{2}.
\end{equation} 
From the previous subsection, we have $$\Upsilon_{K_{p,pr+1}}(t)\geq \Upsilon_K(pt)-\frac{(pr+2)(p-1)t}{2}.$$
Combining this and inequality (3.10), we get $$\Upsilon_{K_{p,q}}(t)\geq \Upsilon_K(pt)-\frac{(p-1)(q+1)t}{2}.$$
The other half of the inequality follows from an analogous argument by considering $\bar{h}(pl+1, t)-\bar{h}(q,t)\geq 0,$ where $l$ is an integer such that $q\leq pl+1$. We omit the details.
\end{proof}

\section{Applications}
\subsection{Computation of $\Upsilon_{(T_{2,-3})_{2,2n+1}}(t)$}
In this subsection, we show how one can compute $\Upsilon_{(T_{2,-3})_{2,2n+1}}(t)$ by using our theorem together with $\widehat{HFK}((T_{2,-3})_{2,2n+1})$, for $n\geq 8$. Note none of these knots is an L-space knot. For easier illustration, we only give full procedure of the computation for the case $K=(T_{2,-3})_{2,17}$. The general case can be done in a similar way. 

By Proposition 4.1 of \cite{Hed05a}, for $i\geq 0$, we have
\begin{displaymath}\widehat{HFK}(K,i)\cong
\begin{cases}
\mathbb{F}_{(2)} &i=10\cr 
\mathbb{F}_{(1)} &i=9 \cr 
\mathbb{F}_{(1)}\oplus \mathbb{F}_{(0)} &i=8 \cr
\mathbb{F}_{(0)}\oplus \mathbb{F}_{(-1)} &i=7\cr
\mathbb{F}_{(i-8)} &0\leq i\leq 6\cr
0 & \text{otherwise}
\end{cases}
\end{displaymath}

Here the subindex stands for the Maslov grading. Note that by using the symmetry $\widehat{HFK}_d(K,i)=\widehat{HFK}_{d-2i}(K,-i)$ \cite{OS04c}, the above equation actually tells us the whole $\widehat{HFK}(K)$.
Now thinking $CFK^{\infty}(K)$ as $\widehat{HFK}(K)\otimes\mathbb{F}[U,U^{-1}]$ when regarded as an $\mathbb{F}[U,U^{-1}]$-module, we see the lattice points supporting generators with Maslov grading $0$ are $(0,7)$, $(7,0)$, $(i,8-i)$, where $-1\leq i \leq 9$. Here, for example, $(0,7)$ means the corresponding generator has algebraic grading $0$ and Alexander grading $7$.

Note by Theorem 1.2 of \cite{Hed09}, $\tau(K)=7$. In view of Theorem 13.1 in \cite{Liv14}, we see that for $t\in[0,\epsilon]$, $\Upsilon_K(t)=-2s$, where $s=\frac{t}{2}\cdot 7+(1-\frac{t}{2})\cdot 0=\frac{7t}{2}$ and $\epsilon$ is sufficiently small. In other words, when $t$ is small, the $\Upsilon_K$ is determined by the $\mathcal{F}_t$ grading of the generator at $(0,7)$. Now by Theorem 7.1 in \cite{Liv14}, singularities of $\Upsilon_K(t)$ can only occur at time $t$ when there is a line of slope $1-\frac{2}{t}$ that contains at least two lattice points supporting generators of Maslov grading $0$. The only $t\in(0,1)$ satisfying this property is $\frac{2}{3}$, giving a line of slope $-2$ that passes through the lattice points $(0,7)$ and $(-1,9)$.

So far, we can see that $\Upsilon_K(t)$ is either one of the two below, depending on whether $\frac{2}{3}$ is a singular point or not.
\begin{equation}\Upsilon_K(t)=
\begin{cases}
-7t, &t\in[0,\frac{2}{3}]\cr 
2-10t, &t\in[\frac{2}{3},1] 
\end{cases}
\end{equation}
Or
\begin{equation}
\Upsilon_K(t)=-7t, \qquad t\in[0,1].
\end{equation}
Note $T_{2,-3}$ is alternating, so we can apply Theorem 1.14 in \cite{OSS} to obtain $\Upsilon_{T_{2,-3}}(t)=1-|1-t|$, when $t\in [0,2]$. Applying Theorem \ref{main} we see that when $\frac{1}{2}\leq t\leq 1$, we have
\begin{displaymath}
2-11t\leq \Upsilon_K(t)\leq 2-10t.
\end{displaymath}
Now we see only (4.1) satisfies this constraint and hence $\Upsilon_K(t)$ is determined.
More generally, we have
\begin{prop}
For $n\geq 8$,
$$\Upsilon_{(T_{2,-3})_{2,2n+1}}(t)=
\begin{cases}
-(n-1)t, &t\in[0,\frac{2}{3}]\cr 
2-(n+2)t, &t\in[\frac{2}{3},1] 
\end{cases}$$
\end{prop}
\begin{proof}
Same as the discussion above. We refer the reader to \cite{Hed05a} for the formula of $\widehat{HFK}((T_{2,-3})_{2,2n+1})$.
\end{proof}
\subsection{An infinite-rank summand of topologically slice knots}
Let $D$ denote the untwisted positive whitehead double of the trefoil knot. Fix an integer $p>2$ and let $J_n=((D_{p,1})...)_{p,1}$ denote the $n$-fold iterated $(p,1)$-cable of $D$ for some positive integer $n$. Recall Corollary \ref{cor} states that $J_n$ for $n=1,2,3,...$ are linearly independent in $\mathcal{C}$ and span an infinite-rank summand consisting of topologically slice knots. To prove this, we first establish two lemmas.
\begin{lem}
Let $\xi_n$ be the first singularity of $\Upsilon_{J_n}(t)$, then $\xi_n \in [\frac{1}{p^n},\frac{2}{1+p^n}]$. In particular, $\xi_i<\xi_j$ whenever $i>j$.  
\end{lem}
\begin{proof}
We first deal with the lower bound. Recall for any knot $K$, $\Upsilon_K(t)=-\tau(K)t$ when $t<\frac{1}{g_3(K)}$ \cite{Liv14}. Note $\tau(D)=g_3(D)=1$ by \cite{Hed07} and hence $\tau(J_n)=p^n$  by \cite{Hed09}. This implies $g_3(J_n)=p^n$ since we have $p^n\leq g_4(J_n)\leq g_3(J_n)\leq p^n$. Therefore, $\xi_n\geq \frac{1}{p^n}$.\\
We move to establish the upper bound. Note $CFK^\infty(D)\cong CFK^\infty(T_{2,3})\oplus A$, where $A$ is an acyclic chain complex \cite{HKL16}. Therefore $\Upsilon_D(t)=\Upsilon_{T_{2,3}}(t)=|1-t|-1$. In particular, $\Upsilon_D(t)= t-2$ when $1\leq t\leq 2$. Apply Theorem \ref{main} we get $\Upsilon_{J_1}(t)\geq pt-2-(p-1)t=t-2$ when $\frac{1}{p}\leq t\leq \frac{2}{p}$. Inductively we have $\Upsilon_{J_n}(t)\geq t-2$ when $\frac{1}{p^n}\leq t\leq \frac{2}{p^n}$. Suppose $\xi_n>\frac{2}{1+p^n}$, then $\exists \epsilon>0$ such that $\Upsilon_{J_n}(\frac{2}{1+p^n}+\epsilon)=-p^n(\frac{2}{1+p^n}+\epsilon)=\frac{-2p^n}{1+p^n}-\epsilon p^n<\frac{2}{1+p^n}-2<\frac{2}{1+p^n}+\epsilon-2$, which is a contradiction. Therefore, $\xi_n\leq \frac{2}{1+p^n}$.
\end{proof}
Let $\Delta\Upsilon'_K(t_0)$ denote the slope change of $\Upsilon_K(t)$ at $t_0$, i.e. $\Delta\Upsilon'_K(t_0)=\lim_{t\searrow t_0}\Upsilon'_K(t_0)-\lim_{t\nearrow t_0}\Upsilon'_K(t_0)$. Recall in general $\frac{t_0}{2}\Delta\Upsilon'_K(t_0)$ is an integer \cite{OSS}. The following lemma shows in some cases, we can determine the value of $\frac{t_0}{2}\Delta\Upsilon'_K(t_0)$.
\begin{lem}
Let $K$ be a knot in $S^3$ such that $\tau(K)\geq 0$ and let $\xi$ be the first singularity of $\Upsilon_K(t)$. If $0<\xi<\frac{4}{g_3(K)+\tau(K)}$, then $\frac{\xi}{2}\Delta\Upsilon'_K(\xi)=1$.
\end{lem}
\begin{proof}
Depicting the chain complex $CFK^\infty(K)$ as lattice points in the plane, by Theorem 7.1 (3) of \cite{Liv14}, we know there is a line of slope $1-\frac{2}{\xi}$ containing at least two lattice points $(i,j)$ and $(i',j')$ supporting generators of Maslov grading $0$. Since $\xi$ is the first singularity, by Theorem 13.1 of \cite{Liv14} we know, say, $(i,j)=(0,\tau(K))$. So we have $\frac{j'-\tau(K)}{i'}=1-\frac{2}{\xi}$, which implies $0<\xi=\frac{2i'}{i'-j'+\tau(K)}<\frac{4}{g_3(K)+\tau(K)}$. This together with the genus bound property of knot Floer homology $|i'-j'|\leq g_3(K)$ would constrain $|i'|=1$. By Theorem 7.1 (4) of \cite{Liv14}, $\frac{\xi}{2}\Delta\Upsilon'_K(\xi)=|i'|=1.$ 
\end{proof}
\begin{proof}[Proof of Corollary \ref{cor}]
Note all $J_n$ have trivial Alexander polynomial and hence are topologically slice \cite{Fre82}. The linear independence follows from Lemma 4.2: suppose $\Sigma k_iJ_{n_i}=0$ in $\mathcal{C}$ for some $k_i\neq 0$ and $n_1>...>n_l$, since $\Upsilon_{\Sigma k_iJ_{n_i}}(t)=\Sigma k_i\Upsilon_{J_{n_i}}(t)$ it possesses first singularity at $\xi_{n_1}$, which contradicts to $\Upsilon_{\Sigma k_iJ_{n_i}}(t)\equiv 0$. To see they span a summand, note by Lemma 4.2 and 4.3, $\frac{\xi_n}{2}\Delta\Upsilon'_{J_n}(\xi_n)=1$. Now one can easily see the homomorphism $\mathcal{C}\longrightarrow \mathbb{Z}^\infty$ given by $[K]\mapsto (\frac{\xi_n}{2}\Delta\Upsilon'_K(\xi_n))_{n=1}^\infty$ is an isomorphism when restricted to the subgroup spanned by $J_n$ and hence the conclusion follows.  
\end{proof}

\begin{rmk}
One can actually replace $D$ by any topologically slice knot $K$ with $\tau(K)=g(K)=1$ and even consider mixed iterated cable $((K_{p_1,1})...)_{p_n,1}$. We chose $J_n$ for the sake of an easier illustration.  The linear independence of certain subfamily of mixed iterated cables of $D$ was also observed by Feller, Park, and Ray \cite{FPR16}.
\end{rmk}

\end{document}